\documentclass[reqno]{amsart}
\usepackage{mathrsfs}
\usepackage{amssymb}
\usepackage{amscd}
\theoremstyle{plain}
\newtheorem{theorem}{Theorem}[section]
\newtheorem{lemma}{Lemma}[section]

\newtheorem{conjecture}{Conjecture}[section]
\newtheorem{example}{Example}[section]

\begin{document}
\title{On Panja-Prasad conjecture}
\thanks{}
\thanks{}
\maketitle
\begin{center}
Qian Chen\\
Department of Mathematics, Shanghai Normal University,
Shanghai 200234, China.\\
Email address: qianchen0505@163.com
\end{center}
\maketitle
\begin{abstract}
In the present paper we shall investigate the Waring's problem for upper triangular matrix algebras. The main result is the following: Let $n\geq 2$ and $m\geq 1$ be integers. Let $p(x_1,\ldots,x_m)$ be a noncommutative polynomial with zero constant term over an infinite field $K$. Let $T_n(K)$ be the set of all $n\times n$ upper triangular matrices over $K$. Suppose $1<r<n-1$, where $r$ is the order of $p$. We have that $p(T_n(K))+p(T_n(K))=J^r$, where $J$ is the Jacobson radical of $T_n(K)$. If $r=n-2$, then $p(T_n(K))=J^{n-2}$. This gives a definitive solution of a conjecture proposed by Panja and Prasad.
\end{abstract}

\section{Introduction}

The classical Waring's problem, proposed by Edward Waring in 1770 and solved
by David Hilbert in 1909, asks whether for every positive integer $k$ there exists a positive integer $g(k)$ such that every positive integer can be expressed as a sum of $g(k)$ $k$th powers of nonnegative integers. Various extensions and variations of this problem have been studied by different groups of mathematicians. For example, in 2011 Larsen, Shalev, and Tiep \cite{Larsen} proved that given a word $w=w(x_1,\ldots,x_m)\neq 1$,
every element in any finite non-abelian simple group $G$ of sufficiently high order can be written as the product of two elements
from $w(G)$, the image of the word map induced by $w$. This is a definitive
solution of the Waring problem for finite simple groups which covers several
partial solutions obtained earlier. In particular, in 2009 Shalev \cite{Shalev} proved
that, under the same assumptions, every element in $G$ is the product of three
elements from $w(G)$.

In 1992, Helmke \cite{Hel} discussed Waring's problem for binary forms. In 2002, Fontanari \cite{Fon} investigated Waring's problem for many forms and Grassmann defective varieties. In 2019, Karabulut \cite{Kar} ontained explicit results for Waring's problem over general finite rings, especially matrix rings over finite fields.

Let $n\geq 2$ be an integer. Let $K$ be a field and let $K\langle X\rangle$ be the free associative algebra over $K$, freely generated by the countable set $X=\{x_1,x_2,\ldots\}$ of
noncommtative variables. We refer to the elements of $K\langle X\rangle$ as
polynomials.

Let $p(x_1,\ldots,x_m)\in K\langle X\rangle$. Let $\mathcal{A}$ be an algebra over $K$. The set
\[
p(\mathcal{A})=\{p(a_1,\ldots,a_m)~|~a_1,\ldots,a_m\in \mathcal{A}\}
\]
is called the image of $p$ (in $\mathcal{A}$).

In 2020, Bre\v{s}ar \cite{Bresar1} initiated the study of various Waring's problems for matrix algebras. One of the main results \cite[Theorem 3.18]{Bresar1}, states that if $\mathcal{A}=M_n(K)$,
where $n\geq 2$ and $K$ is an algebraically closed field with characteristic $0$,
and $f$ is a noncommutative polynomial which is neither an identity nor a central polynomial of $\mathcal{A}$, then
$f(\mathcal{A})-f(\mathcal{A})$ contains all square-zero matrices in $\mathcal{A}$. Using the fact that every trace zero
matrix is a sum of four square-zero matrices \cite{Pazzis}, it follows that every trace
zero matrix in $\mathcal{A}$ is a sum of four matrices from $f(\mathcal{A})-f(\mathcal{A})$ \cite[Corollary 3.19]{Bresar1}. Furthermore, Bre\v{s}ar and \v{S}emrl \cite{B1} proved that any traceless matrix can be written as sum of two
matrices from $f(M_n(\mathcal{C}))-f(M_n(\mathcal{C}))$, where $\mathcal{C}$ is the complex field and $f$ is neither an identity nor a central polynomial for
$M_n(\mathcal{C})$. Recently, they have proved that if $\alpha_1, \alpha_2,\alpha_3\in \mathcal{C}\setminus\{0\}$ and
$\alpha_1+\alpha_2+\alpha_3=0$, then any traceless matrix over $\mathcal{C}$ can be written as $\alpha_1A_1+\alpha_2A_2+\alpha_3A_3$, where $A_i\in f(M_n(\mathcal{C}))$ (see \cite{B2}).

By $T_n(K)$ we denote the set of all $n\times n$ upper triangular
matrices over $K$. By $T_n(K)^{(0)}$ we denote the set of all
$n\times n$ strictly upper triangular matrices over $K$. More generally, if $t\geq 0$, the set of all upper triangular
matrices whose entries $(i,j)$ are zero, for $j- i\leq t$, will be
denoted by $T_n(K)^{(t)}$. It is easy to check that $J^t=T_n(K)^{(t-1)}$, where $t\geq 1$ and $J$ is the Jacobson radical of $T_n(K)$ (see \cite[Example 5.58]{Bre1}).

Let $p(x_1,\ldots,x_m)$ be a noncommutative polynomial with zero constant term over $K$. We define its \textbf{order}
as the least positive integer $r$ such that $p(T_r(K))=\{0\}$ but $p(T_{r+1}(K))\neq\{0\}$. Note that $T_1(K)=K$. We say that $p$ has order $0$ if $p(K)\neq\{0\}$. We denote the order of $p$ by ord$(p)$.  For a detailed introduction of the order of polynomials we refer the reader to the book \cite[Chapter 5]{Drensky}.

In 2023, Panja and Prasad \cite{PP} initiated Waring's problems for upper triangular matrix algebras. More precisely, they obtained the following result:

\begin{theorem}\cite[Theorem 5.18(iii)]{PP}\label{PP}
Let $n\geq 2$ and $m\geq 1$ be integers. Let $p(x_1,\ldots,x_m)$ be a
polynomial with zero constant term in non-commutative variables over an algebraically closed field $K$. Set $r=\emph{ord}(p)$. If $1<r<n-1$, then $p(T_n(K))\subseteq T_n(K)^{(r-1)}$, and equality might not hold in general. Furthermore, for every $n$ and $r$ there exists $d$ such that each element of $T_n(K)^{(r-1)}$ can be written as a sum of $d$ many elements from $p(T_n(K))$
\end{theorem}

They proposed the following conjecture:

\begin{conjecture}\cite[Conjecture]{PP}\label{PP2}
Let $p(x_1,\ldots,x_m)$ be a noncommutative polynomial with zero constant term over an algebraically closed field $K$. Suppose ord$(p)=r$, where $1<r<n-1$. Then
$p(T_n(K))+p(T_n(K))=T_n(K)^{(r-1)}$.
\end{conjecture}

We remark that ``$p(T_n(K))+p(T_n(K))=T_n(K)^{(r-1)}$" can be replaced by ``$p(T_n(K))=T_n(K)^{(r-1)}$" in Theorem \ref{PP2} in the case of $p$ is a linear polynomial (in particular, multilinear polynomial) and $K$ is an infinite field (see \cite{Mello, Luo, Wang2022}).

In the present paper, we shall prove the following main result of the paper, which gives a definitive solution of Conjecutue \ref{PP2}.

\begin{theorem}\label{T1}
Let $n\geq 2$ and $m\geq 1$ be integers. Let $p(x_1,\ldots,x_m)$ be a noncommutative polynomial with zero constant term over an infinite field $K$. Suppose ord$(p)=r$, where $1<r<n-1$. We have that $p(T_n(K))+p(T_n(K))=T_n(K)^{(r-1)}$. In particular, if $r=n-2$, then $p(T_n(K))=T_n(K)^{(n-3)}$.
\end{theorem}

We organize the paper as follows: In Section $2$ we shall give some prelimiaries. We shall slightly modify some results in \cite{Chen2023,FK,Jac}, which will be used in the proof of Theorem \ref{T1}. In Section $3$ we shall give the proof of Theorem \ref{T1} by using some new arguments, for example, the compatibility of variables and recursive polynomials. In the end of the paper we give an example, which extends \cite[Example 5.7]{PP}.

\section{prelimiaries}

Let $\mathcal{N}$ be the set of all positive integers. Let $m\in \mathcal{N}$. Let $K$ be a field. Set $K^*=K\setminus\{0\}$. For any $k\in \mathcal{N}$ we set
\[
T_k=\left\{(i_1,\ldots,i_k)\in \mathcal{N}^k~|~1\leq i_1,\ldots,i_k\leq m\right\}.
\]

Let $p(x_1,\ldots,x_m)$ be a noncommutative polynomial with zero constant term over
$K$. We can write
\begin{equation}\label{e1}
p(x_1,\ldots,x_m)=\sum\limits_{k=1}^d\left(\sum\limits_{(i_1,i_2,\ldots,i_k)\in T_k}\lambda_{i_1i_2\cdots i_k}x_{i_1}x_{i_2}\cdots x_{i_k}\right),
\end{equation}
where $\lambda_{i_1i_2\cdots i_k}\in K$ and $d$ is the degree of $p$.

We begin with the following result, which is slightly different from \cite[Lemma 3.2]{Chen2023}. We give its proof for completeness.

\begin{lemma}\label{M1}
For any $u_i=(a_{jk}^{(i)})\in T_n(K)$, $i=1,\ldots,m$, we set
\[
\bar{a}_{jj}=(a_{jj}^{(1)},\ldots,a_{jj}^{(m)}),
\]
where $j=1,\ldots,n$. We have that

\begin{equation}\label{e2}
p(u_1,\ldots,u_m)=\left(
\begin{array}{cccc}
p(\bar{a}_{11}) & p_{12} & \ldots & p_{1n}\\
0 & p(\bar{a}_{22}) & \ldots & p_{2n}\\
 \vdots & \vdots & \ddots & \vdots\\
 0 & 0 & \ldots & p(\bar{a}_{nn})
\end{array} \right),
\end{equation}
where
\[
p_{st}=\sum\limits_{k=1}^{t-s}\left(\sum\limits_{\substack{s=j_1<j_2<\cdots <j_{k+1}=t\\(i_1,\ldots,i_k)\in T_k}}
p_{i_1\cdots i_k}(\bar{a}_{j_1j_1},\ldots,\bar{a}_{j_{k+1}j_{k+1}})a_{j_1j_2}^{(i_1)}\cdots a_{j_{k}j_{k+1}}^{(i_k)}\right)
\]
for all $1\leq s\leq t\leq n$, where $p_{i_1,\ldots,i_k}(z_1,\ldots,z_{m(k+1)})$, $1\leq i_1,i_2,\ldots,i_k\leq m$, $k=1,\ldots,n-1$, is a
commutative polynomial over $K$.
\end{lemma}

\begin{proof}
Let $u_i=(a_{jk}^{(i)})\in T_n(K)$, where $i=1,\ldots,m$. For any $1\leq i_1,\ldots,i_k\leq m$, we easily check that
\[
u_{i_1}\cdots u_{i_k}=\left(
\begin{array}{cccc}
m_{11} & m_{12} & \ldots & m_{1n}\\
0 & m_{22} & \ldots & m_{2n}\\
 \vdots & \vdots & \ddots & \vdots\\
 0 & 0 & \ldots & m_{nn}
\end{array} \right),
\]
where
\[
m_{st}=\sum\limits_{s=j_1\leq j_2\leq \cdots \leq j_{k+1}=t}a_{j_1j_2}^{(i_1)}\cdots a_{j_{k}j_{k+1}}^{(i_k)}
\]
for all $1\leq s\leq t\leq n$. It follows from (\ref{e1}) that
\begin{eqnarray*}
\begin{split}
p(u_1,\ldots,u_m)&=\sum\limits_{k=1}^d\left(\sum\limits_{(i_1,\ldots,i_k)\in T_k}\lambda_{i_1\cdots i_k}u_{i_1}\cdots u_{i_k}\right)\\
&=\sum\limits_{k=1}^d\left(\sum\limits_{(i_1,\ldots,i_k)\in T_k}\lambda_{i_1\cdots i_k}\left(
\begin{array}{cccc}
m_{11} & m_{12} & \ldots & m_{1n}\\
0 & m_{22} & \ldots & m_{2n}\\
 \vdots & \vdots & \ddots & \vdots\\
 0 & 0 & \ldots & m_{nn}
\end{array} \right)\right)\\
&=\left(
\begin{array}{cccc}
p_{11} & p_{12} & \ldots & p_{1n}\\
0 & p_{22} & \ldots & p_{2n}\\
 \vdots & \vdots & \ddots & \vdots\\
 0 & 0 & \ldots & p_{nn}
\end{array} \right)
\end{split}
\end{eqnarray*}
where
\begin{eqnarray*}
\begin{split}
p_{st}&=\sum\limits_{k=1}^d\left(\sum\limits_{(i_1,\ldots,i_k)\in T_k}\lambda_{i_1\cdots i_k}m_{st}\right)\\
&=\sum\limits_{k=1}^d\left(\sum\limits_{(i_1,\ldots,i_k)\in T_k}\lambda_{i_1\cdots i_k}\left(
\sum\limits_{s=j_1\leq j_2\leq \cdots \leq j_{k+1}=t}a_{j_1j_2}^{(i_1)}\cdots a_{j_{k}j_{k+1}}^{(i_k)}\right)\right)\\
&=\sum\limits_{k=1}^{d}\left(\sum\limits_{\substack{s=j_1\leq j_2\leq \cdots \leq j_{k+1}=t\\(i_1,\ldots,i_k)\in T_k}}\lambda_{i_1i_2\cdots i_k}a_{j_1j_2}^{(i_1)}\cdots a_{j_{k}j_{k+1}}^{(i_k)}\right),
\end{split}
\end{eqnarray*}
where $1\leq s\leq t\leq n$. In particular
\begin{eqnarray*}
\begin{split}
p_{ss}&=\sum\limits_{k=1}^{d}\left(\sum\limits_{(i_1,\ldots,i_k)\in T_k}\lambda_{i_1i_2\cdots i_k}a_{ss}^{(i_1)}\cdots a_{ss}^{(i_k)}\right)\\
&=p(\bar{a}_{ss})
\end{split}
\end{eqnarray*}
for all $s=1,\ldots,n$, and
\begin{eqnarray*}
\begin{split}
p_{st}&=\sum\limits_{k=1}^{d}\left(\sum\limits_{\substack{s=j_1\leq j_2\leq \cdots \leq j_{k+1}=t\\(i_1,\ldots,i_k)\in T_k}}\lambda_{i_1i_2\cdots i_k}a_{j_1j_2}^{(i_1)}\cdots a_{j_{k}j_{k+1}}^{(i_k)}\right)\\
&=\sum\limits_{k=1}^{t-s}\left(\sum\limits_{\substack{s=j_1<j_2<\cdots <j_{k+1}=t\\(i_1,\ldots,i_k)\in T_k}}
p_{i_1i_2\cdots i_k}(\bar{a}_{j_1j_1},\ldots,\bar{a}_{j_{k+1}j_{k+1}})a_{j_1j_2}^{(i_1)}\cdots a_{j_{k}j_{k+1}}^{(i_k)}\right)
\end{split}
\end{eqnarray*}
for all $1\leq s<t\leq n$, where $p_{i_1,\ldots,i_k}(z_1,\ldots,z_{m(k+1)})$ is a commutative polynomial over $K$. This proves the result.
\end{proof}

The following result is slightly different from \cite[Lemma 3.4]{Chen2023}. We give its proof for completeness.

\begin{lemma}\label{L2.1}
Let $m\geq 1$ be an integer. Let $p(x_1,\ldots,x_m)$ be a polynomial with zero constant term over a field $K$.
Let $p_{i_1,\ldots,i_k}(z_1,\ldots,z_{m(k+1)})$ be a commutative polynomial of $p$ in (\ref{e2}), where $1\leq i_1,\ldots,i_k\leq m$, $1\leq k\leq n-1$. Suppose that \emph{ord}$(p)=r$, $1<r<n-1$. We have that
\begin{enumerate}
\item[(i)] $p(K)=\{0\}$;
\item[(ii)] $p_{i_1,\ldots,i_k}(K)=\{0\}$ for all $1\leq i_1,\ldots,i_k\leq m$, where $k=1,\ldots,r-1$. This implies that $p(T_n(K))\subseteq T_n(K)^{(r-1)}$;
\item[(iii)] $p_{i_1',\ldots,i_r'}(K)\neq\{0\}$ for some $1\leq i_1',\ldots,i_r'\leq m$.
\end{enumerate}
\end{lemma}

\begin{proof}
It is clear that the statement (i) holds true. We now claim that the statement (ii) holds true. Suppose on the contrary that
\[
p_{i_1'\cdots i_s'}(K)\neq \{0\}
\]
for some $1\leq i_1',\ldots,i_s'\leq m$, where $1\leq s\leq r-1$. Then there exist $\bar{b}_{j}\in K^m$, where $j=1,\ldots,s+1$ such that
\[
p_{i_1'\cdots i_s'}(\bar{b}_{1},\ldots,\bar{b}_{s+1})\neq 0.
\]
We take $u_i=(a_{jk}^{(i)})\in T_{s+1}(K)$, $i=1,\ldots,m$, where
\[
\left\{
\begin{aligned}
\bar{a}_{jj}&=\bar{b}_{j},\quad\mbox{$j=1,\ldots,s+1$};\\
a_{k,k+1}^{(i_k')}&=1,\quad\mbox{$k=1,\ldots,s$};\\
a_{jk}^{(i)}&=0,\quad\mbox{otherwise}.
\end{aligned}
\right.
\]
It follows from (\ref{e2}) that
\[
p_{1,s+1}=p_{i_1'\cdots i_s'}(\bar{b}_{1},\ldots,\bar{b}_{s+1})\neq 0.
\]
This implies that $p(T_{s+1}(K))\neq\{0\}$, a contradiction. This proves the statement (ii).

We finally claim that the statement (iii) holds true. Note that $p(T_{1+r}(K))\neq\{0\}$. Thus,
we have that there exist $u_{i}=(a_{jk}^{(i)})\in T_{1+r}(K)$, $i=1,\ldots,m$, such that
\[
p(u_1,\ldots,u_m)=(p_{st})\neq 0.
\]
In view of the statement (ii) we get that
\[
p_{1,r+1}=\sum\limits_{\substack{1=j_1<j_2< \cdots <j_{r+1}=r+1\\(i_1,\ldots,i_r)\in T_k}}p_{i_1i_2\cdots i_r}(\bar{a}_{j_1j_1},\ldots,\bar{a}_{j_{r+1}j_{r+1}})a_{j_1j_2}^{(i_1)}\cdots a_{j_{r}j_{r+1}}^{(i_r)}\neq 0.
\]
This implies that $p_{i_1',\ldots,i_r'}(K)\neq \{0\}$ for some $1\leq i_1',\ldots,i_r'\leq m$. This proves the statement (iii). The proof of the result is complete.
\end{proof}

The following well known result will be used in the proof of the rest results.

\begin{lemma}\cite[Theorem 2.19]{Jac}\label{L2.2}
Let $K$ be an infinite field. Let $f(x_1,\ldots,x_m)$ be a nonzero commutative polynomial over $K$. Then there exist $a_1,\ldots,a_m\in K$ such that $f(a_1,\ldots,a_m)\neq 0$.
\end{lemma}

The following result is slightly defferent from \cite[Lemma 3.5]{Chen2023}. We give its proof for completeness.

\begin{lemma}\label{L2.3}
Let $n,s$ be integers with $1\leq s\leq n$. Let $p(x_1,\ldots,x_s)$ be a nonzero commutative polynomials over an infinite field $K$.  We have that there exist $a_1,\ldots,a_n\in K$ such that
\begin{eqnarray*}
\begin{split}
p(a_{i_1},\ldots,a_{i_s})\neq 0
\end{split}
\end{eqnarray*}
for all $1\leq i_1<\cdots<i_s\leq n$.
\end{lemma}

\begin{proof}
We set
\[
f(x_1,\ldots,x_n)=\prod_{1\leq i_1<\cdots<i_s\leq n}p(x_{i_1},\ldots,x_{i_s}).
\]
It is clear that $f\neq 0$. In view of Lemma \ref{L2.2} we have that there exist $a_1,\ldots,a_n\in K$ such that
\[
f(a_1,\ldots,a_n)\neq 0.
\]
This implies that
\[
p(a_{i_1},\ldots,a_{i_s})\neq 0
\]
for all $1\leq i_1<\cdots<i_s\leq n$. This proves the result.
\end{proof}

The following technical result is a generalization of \cite[Lemma 2.11]{FK}. We give its proof for completeness.

\begin{lemma}\label{L2.4}
Let $t\geq 1$. Let $U_i=\{i_1,\ldots,i_s\}$, $i=1,\ldots,t$. Let $p_i(x_{i_1},\ldots,x_{i_s})$ be a nonzero commutative polynomial over an infinite field $K$, where $i=1,\ldots,t$. Then there exist $a_{k}\in K$ with $k\in\bigcup_{i=1}^tU_i$ such that
\[
p_i(a_{i_1},\ldots,a_{i_s})\neq 0
\]
for all $i=1,\ldots,t$.
\end{lemma}

\begin{proof}
Without loss of generality we assume that
\[
\{1,2,\ldots,n\}=\bigcup_{i=1}^tU_i.
\]
We set
\[
f(x_1,\ldots,x_n)=\prod_{i=1}^{t}p_i(x_{i_1},\ldots,x_{i_s}).
\]
It is clear that $f\neq 0$. In view of Lemma \ref{L2.2} we have that there exist $a_1,\ldots,a_n\in K$ such that
\[
f(a_{1},\ldots,a_n)\neq 0.
\]
This implies that
\[
p_i(a_{i_1},\ldots,a_{i_s})\neq 0
\]
for all $i=1,\ldots,t$. This proves the result.
\end{proof}

\section{The proof of Theorem \ref{T1}}

We begin with the following technical result, which is curcial for the proof of the main result of the paper.

\begin{lemma}\label{L3.1}
Let $s\geq 1$ and $t\geq 2$ be integers. Let $K$ be an infinite field.
Let $a_{i1},\ldots,a_{is},b\in K$, $i=1,\ldots,t$, where $b,a_{11}\in K^*$. For any $i\in\{2,\ldots,t\}$, there exists a nonzero element in $\{a_{i1},\ldots,a_{is}\}$. Then there exist $c_i\in K$, $i=1,\ldots,s$, such that
\[
\left\{
\begin{aligned}
a_{11}c_1+\cdots +a_{1s}c_s&=b;\\
a_{i1}c_1+\cdots +a_{is}c_s&\neq 0
\end{aligned}
\right.
\]
for all $i=2,\ldots,t$.
\end{lemma}

\begin{proof}
The case of $s=1$ is obvious. We now assume that $s\geq 2$. Suppose first that $a_{i1}\neq 0$ for all $i=2,\ldots,t$. We define the following polynomials.
\[
\left\{
\begin{aligned}
f_1(x_2,\ldots,x_s)&=b-a_{12}x_2-\cdots -a_{1s}x_s;\\
f_i(x_2,\ldots,x_s)&=a_{i1}a_{11}^{-1}b+(a_{i2}-a_{i1}a_{11}^{-1}a_{12})x_2+\cdots +(a_{is}-a_{i1}a_{11}^{-1}a_{1s})x_s
\end{aligned}
\right.
\]
for all $i=2,\ldots,t$. Since $b,a_{i1}\in K^*$, $i=1,\ldots,t$, we note that $f_i\neq 0$ for all $i=1,\ldots,t$. In view of Lemma \ref{L2.4} we get that there exist $c_2,\ldots,c_s\in K$ such that
\[
f_i(c_2,\ldots,c_s)\neq 0
\]
for all $i=1,\ldots,t$. This implies that
\begin{equation}\label{Chen1}
\left\{
\begin{aligned}
b-a_{12}c_2-\cdots -a_{1s}c_s&\neq 0;\\
a_{i1}a_{11}^{-1}b+(a_{i2}-a_{i1}a_{11}^{-1}a_{12})c_2+\cdots +(a_{is}-a_{i1}a_{11}^{-1}a_{1s})c_s&\neq 0
\end{aligned}
\right.
\end{equation}
for all $i=2,\ldots,t$. We set
\[
c_1=a_{11}^{-1}(b-a_{12}c_2-\cdots -a_{1s}c_s).
\]
It follows from (\ref{Chen1}) that
\[
\left\{
\begin{aligned}
a_{11}c_1+\cdots +a_{1s}c_s&=b;\\
a_{i1}c_1+\cdots +a_{is}c_s&\neq 0
\end{aligned}
\right.
\]
for all $i=2,\ldots,t$, as desired.

Suppose next that $a_{i1}=0$, $i=2,\ldots,t$. Note that $a_{il(i)}\neq 0$, for some $2\leq l(i)\leq s$ for all $i=2,\ldots,t$. We define the following polynomials:
\[
\left\{
\begin{aligned}
f_1(x_2,\ldots,x_s)&=a_{12}x_2+\cdots +a_{1s}x_s-b;\\
f_i(x_2,\ldots,x_s)&=a_{i2}x_2+\cdots +a_{is}x_s
\end{aligned}
\right.
\]
for all $i=2,\ldots,t$. Note that $f_i\neq 0$ for all $i=1,\ldots,t$. In view of Lemma \ref{L2.4} we get that there exist $c_i\in K$, $i=2,\ldots,s$, such that
\[
f_i(c_2,\ldots,c_s)\neq 0
\]
for all $i=1,\ldots,t$. That is
\[
\left\{
\begin{aligned}
&a_{12}c_2+\cdots +a_{1s}c_s-b\neq 0;\\
&a_{i2}c_2+\cdots +a_{is}c_s\neq 0
\end{aligned}
\right.
\]
for all $i=1,\ldots,t$. Since $a_{11}\neq 0$ we get that there exists $c_{1}\in K$ such that
\[
a_{11}c_1=b-a_{12}c_2-\cdots -a_{1s}c_s.
\]
This implies that
\[
\left\{
\begin{aligned}
a_{11}c_1+&a_{12}c_2+\cdots +a_{1s}c_s=b;\\
&a_{i2}c_2+\cdots +a_{is}c_s\neq 0
\end{aligned}
\right.
\]
as desired.

We finally assume that there exist $a_{i1}\neq 0$ and $a_{j1}=0$ for some $i,j\in \{2,\ldots,t\}$. Without loss of generality we assume that
$a_{i1}\neq 0$ for all $i=2,\ldots,t_1$ and $a_{i1}=0$ for all $i=t_1+1,\ldots,t$. We define the following polynomials:
\[
\left\{
\begin{aligned}
f_1(x_2,\ldots,x_s)&=b-a_{12}x_2-\cdots -a_{1s}x_s;\\
f_{i}(x_2,\ldots,x_s)&=a_{i1}a_{11}^{-1}b+(a_{i2}-a_{i1}a_{11}^{-1}a_{12})x_2+\cdots +(a_{is}-a_{i1}a_{11}^{-1}a_{1s})x_s;\\
g_j(x_2,\ldots,x_s)&=a_{j2}x_2+\cdots +a_{js}x_s
\end{aligned}
\right.
\]
where $i=2,\ldots, t_1$ and $j=t_1+1,\ldots,t$. Note that $b,a_{i1}\in K^*$, $i=1,\ldots,t_1$, $a_{jl(j)}\neq 0$ where $2\leq l(j)\leq s$ for all $j=t_1+1,\ldots t$. We get that $f_i,g_j\neq 0$ for all $i=1,\ldots,t_1$ and $j=t_1+1,\ldots,t$. In view of Lemma \ref{L2.4} we get that there exist $c_i\in K$, $i=2,\ldots,s$, such that
\[
\left\{
\begin{aligned}
f_1(c_2,\ldots,c_s)&\neq 0;\\
f_{i}(c_2,\ldots,c_s)&\neq 0;\\
g_j(c_2,\ldots,c_s)&\neq 0,
\end{aligned}
\right.
\]
where $i=1,\ldots, t_1$ and $j=t_1+1,\ldots,t$. This implies that
\begin{equation}\label{Chen2}
\left\{
\begin{aligned}
b-a_{12}c_2-\cdots -a_{1s}c_s&\neq 0;\\
a_{i1}a_{11}^{-1}b+(a_{i2}-a_{i1}a_{11}^{-1}a_{12})c_2+\cdots +(a_{is}-a_{i1}a_{11}^{-1}a_{1s})c_s&\neq 0;\\
a_{j2}c_2+\cdots +a_{js}c_s&\neq 0,
\end{aligned}
\right.
\end{equation}
where $i=2,\ldots, t_1$ and $j=t_1+1,\ldots,t$. We set
\[
c_1=a_{11}^{-1}(b-a_{12}c_2-\cdots -a_{1s}c_s).
\]
It follows from (\ref{Chen2}) that
\[
\left\{
\begin{aligned}
a_{11}c_1+\cdots +a_{1s}c_s&=b;\\
a_{i1}c_1+\cdots +a_{is}c_s&\neq 0;\\
a_{j1}c_2+\cdots +a_{js}c_s&\neq 0,
\end{aligned}
\right.
\]
where $i=2,\ldots,t_1$ and $j=t_1+1,\ldots,t$, as desired. The proof of the result is now complete.
\end{proof}

The following result is crucial for the proof of the main result of the paper.

\begin{lemma}\label{L3.2}
Let $p(x_1,\ldots,x_m)$ is a noncommutative polynomial with zero constant term over an infinite field $K$. Suppose ord$(p)=r$, where $1<r<n-1$. For any $A'=(a_{s,r+s+t}')\in T_n(K)^{(r-1)}$, where $a_{s,r+s}'\neq 0$ for all $1\leq s<r+s\leq n$. Then $A'\in p(T_n(K))$.
\end{lemma}

\begin{proof}
For any $u_i=(a_{jk}^{(i)})\in T_n(K)$, $i=1,\ldots,m$, in view of both Lemma \ref{M1} and Lemma \ref{L2.1} that

\begin{equation}\label{e3}
p(u_1,\ldots,u_m)=(p_{s,r+s+t})
\end{equation}
where
\[
p_{s,r+s+t}=\sum\limits_{k=r}^{r+t}\left(\sum\limits_{\substack{s=j_1<\cdots <j_{k+1}=r+s+t\\(i_1,\ldots,i_k)\in T_k}}
p_{i_1\cdots i_k}(\bar{a}_{j_1j_1},\ldots,\bar{a}_{j_{k+1}j_{k+1}})a_{j_1j_2}^{(i_1)}\cdots a_{j_{k}j_{k+1}}^{(i_k)}\right)
\]
for all $1\leq s<r+s+t\leq n$ and
\[
p_{i_1'\cdots i_r'}(K)\neq \{0\}
\]
for some $1\leq i_1',\ldots,i_r'\leq m$. It follows from Lemma \ref{L2.3} that
there exist $\bar{c}_1,\ldots,\bar{c}_n\in K^m$ such that
\begin{equation}\label{e4}
p_{i_1'\cdots i_r'}(\bar{c}_{j_1},\ldots,\bar{c}_{j_{r+1}})\neq 0
\end{equation}
for all $1\leq j_1<\ldots <j_{r+1}\leq n$.

We set
\[
\left\{
\begin{aligned}
\bar{a}_{jj}&=\bar{c}_j,\quad\mbox{$j=1,\ldots,n$};\\
a_{i,i+1}^{(k)}&=a_{i,i+1}^{(k)},\quad\mbox{$i=1,\ldots,r-1$ and $k=1,\ldots,m$};\\
a_{r+s-1,r+s+t}^{(i_k')}&=x_{r+s-1,r+s+t}^{(i_k')},\quad\mbox{$1\leq s<r+s+t\leq n$, $k=1,\ldots,r$};\\
a_{ij}^{(k)}&=0,\quad\mbox{otherwise}.
\end{aligned}
\right.
\]

We define an order of the set of the variables
\[
\left\{x_{r+s-1,r+s+t}^{(i_k')}~|~\quad\mbox{$1\leq s<r+s+t\leq n$, $k=1,\ldots,r$}\right\}
\]
as follows:
\begin{enumerate}
\item[(i)] $x_{r+s-1,r+s+t}^{(i_j')}$ and $x_{r+s-1,r+s+t}^{(i_k')}$ have the same order for all $j,k=1,\ldots,r$;
\item[(ii)] $x_{r+s-1,r+s+t}^{(i_k')}<x_{r+s'-1,r+s'+t'}^{(i_k')}$ for all $t=t'$ and $s<s'$;
\item[(iii)] $x_{r+s-1,r+s+t}^{(i_k')}<x_{r+s'-1,r+s'+t'}^{(i_k')}$ for all $t<t'$.
\end{enumerate}

We set
\[
\hat{c}_{s,t}=(\bar{c}_s,\ldots,\bar{c}_{r+s-1},\bar{c}_{r+s+t}).
\]
It follows from (\ref{e4}) that
\begin{equation}\label{e5}
p_{i_1'\cdots i_r'}(\hat{c}_{s,t})\neq 0.
\end{equation}

For any $1\leq s<r+s\leq n$ and $s\leq r-1$, we set
\[
f_{s,r}=\sum\limits_{(i_1,\ldots,i_{r-s})\in T_{r-s}}p_{i_1\cdots i_{r-s}i_{r-s+1}'\cdots i_r'}(\hat{c}_{s,t})a_{s,s+1}^{(i_1)}\cdots a_{r-1,r}^{(i_{r-s})}.
\]
For any $1\leq s<r+s\leq n$ and $s\geq r$, we set
\[
f_{s,r}=p_{i_1'\cdots i_{r}'}(\hat{c}_{s,t}).
\]

We claim that $f_{s,r}\neq 0$ for all $1\leq s<r+s\leq n$. In view of (\ref{e5}) it suffices to prove that $f_{s,r}\neq 0$, where $1\leq s<r+s\leq n$ and $s\leq r-1$. We set
\[
V_{s,r}=\{(i,i+1,k)~|~i=s,\ldots,r-1,\quad k=1,\ldots,m\}.
\]
It is clear that $f_{s,r}$ is a commutative polynomial over $K$ on the set of variables
\[
\left\{a_{i,i+1}^{(k)}~|~(i,i+1,k)\in V_{s,r}\right\}.
\]

We take $a_{i,i+1}^{(k)}\in K$, $(i,i+1,k)\in V_{s,r}$ such that
\[
\left\{
\begin{aligned}
a_{s+i,s+i+1}^{(i_{i+1}')}&=1\quad\mbox{$i=0,\ldots,r-s-1$};\\
a_{i,i+1}^{(k)}&=0\quad\mbox{otherwise}.
\end{aligned}
\right.
\]
It follows from (\ref{e5}) that
\[
f_{s,r}(a_{i,i+1}^{(k)})=p_{i_1'\cdots i_r'}(\hat{c}_{s,t})\neq 0,
\]
as desired. In view of Lemma \ref{L2.4} we get that there exist $a_{i,i+1}^{(k)}\in K$, $(i,i+1,k)\in \bigcup_{s=1}^{min\{n-r,r-1\}}V_{s,r}$ such that
\[
f_{s,r}(a_{i,i+1}^{(k)})\neq 0
\]
for all $1\leq s<r+s\leq n$ and $s\leq r-1$.

For any $2\leq s\leq r+s\leq n$, we define
\begin{equation}\label{e6}
f_{s,r+s-i}=\sum\limits_{(i_1,\ldots,i_{r-i})\in T_{r-i}}
p_{i_1\cdots i_{r-i}i_{r-i+1}'\cdots i_r'}(\hat{c}_{s,t})a_{s,s+1}^{(i_1)}\cdots a_{r+s-i-1,r+s-i}^{(i_{r-i})}
\end{equation}
for all $1\leq i\leq min\{s-1,r-1\}$. We now claim that (\ref{e6}) is a recursive polynomial over $K$. Indeed, we get from (\ref{e6}) that
\begin{eqnarray}\label{oo1}
\begin{split}
f&_{s,r+s-i}=\left(\sum\limits_{(i_1,\ldots,i_{r-i-1})\in T_{r-i-1}}
p_{i_1\cdots i_{r-i-1}i_{r-i}'\cdots i_r'}(\hat{c}_{s,t})a_{s,s+1}^{(i_1)}\cdots a_{r+s-i-2,r+s-i-1}^{(i_{r-i-1})}\right)x_{r+s-i-1,r+s-i}^{(i_{r-i}')}\\
&=\sum\limits_{\substack{1\leq k\leq r\\i_k'\neq i_{r-i}'}}\left(\sum\limits_{(i_1,\ldots,i_{r-i-1})\in T_{r-i-1}}
p_{i_1\cdots i_{r-i-1}i_k'i_{r-i+1}'\cdots i_r'}(\hat{c}_{s,t})a_{s,s+1}^{(i_1)}\cdots a_{r+s-i-2,r+s-i-1}^{(i_{r-i-1})}\right)x_{r+s-i-1,r+s-i}^{(i_{k}')}
\end{split}
\end{eqnarray}
for all $1\leq i\leq min\{s-1,r-1\}$. We set
\[
\alpha_{s,r+s-i-1,k}=\sum\limits_{(i_1,\ldots,i_{r-i-1})\in T_{r-i-1}}
p_{i_1\cdots i_{r-i-1}i_k'i_{r-i+1}'\cdots i_r'}(\hat{c}_{s,t})a_{s,s+1}^{(i_1)}\cdots a_{r+s-i-2,r+s-i-1}^{(i_{r-i-1})}
\]
for all $1\leq i\leq min\{s-1,r-1\}$ and $k=1,\ldots,r$. It follows from both (\ref{e6}) and (\ref{oo1}) that

\begin{equation}\label{e7}
f_{s,r+s-i}=f_{s,r+s-i-1}x_{r+s-i-1,r+s-i}^{(i_{r-i}')}+\sum\limits_{\substack{1\leq k\leq r\\i_k'\neq i_{r-i}'}}\alpha_{s,r+s-i-1,k}x_{r+s-i-1,r+s-i}^{(i_{k}')}
\end{equation}
for all $1\leq i\leq min\{s-1,r-1\}$, where $f_{s,r+s-i-1}$ and $\alpha_{s,r+s-i-1,k}$ are commutative polynomials over $K$ on some variables in
the front of $x_{r+s-i-1,r+s-i}^{(i_{r-i}')}$.

For any $1\leq s<r+s+t\leq n$, we set
\[
U_{s,r+s+t}=\left\{(r+u-1,r+u+w,i_k')~|~x_{r+u-1,r+u+w}^{(i_k')}\quad\mbox{in $p_{s,r+s+t}$}\right\}
\]
and
\[
\overline{U}_{s,r+s+t}=\left\{(r+u-1,r+u,i_k')\in U_{s,r+s+t}\right\}.
\]

It follows from (\ref{e3}) that

\begin{eqnarray}\label{e8}
\begin{split}
p&_{s,r+s+t}=\left(\sum\limits_{(i_1,\ldots,i_{r-1})\in T_{r-1}}p_{i_1\cdots i_{r-1}i_r'}(\hat{c}_{s,t})a_{s,s+1}^{(i_1)}\cdots a_{r+s-2,r+s-1}^{(i_{r-1})}\right)x_{r+s-1,r+s+t}^{(i_r')}\\
&+\sum\limits_{\substack{1\leq k\leq r\\i_k'\neq i_r'}}\left(\sum\limits_{(i_1,\ldots,i_{r-1})\in T_{r-1}}p_{i_1\cdots i_{r-1}i_k'}(\hat{c}_{s,t})a_{s,s+1}^{(i_1)}\cdots a_{r+s-2,r+s-1}^{(i_{r-1})}\right)x_{r+s-1,r+s+t}^{(i_k')}\\
&+\sum\limits_{k=r}^{r+t}\left(\sum\limits_{\substack{s=j_1<\cdots <j_{k+1}=r+s+t\\(j_k,j_{k+1})\neq (r+s-1,r+s+t)\\(i_1,\ldots,i_k)\in T_k}}
p_{i_1\cdots i_k}(\bar{c}_{j_1},\ldots,\bar{c}_{j_{k+1}})a_{j_1j_2}^{(i_1)}\cdots a_{j_{k}j_{k+1}}^{(i_k)}\right)
\end{split}
\end{eqnarray}

We set

\[
\beta_{s,r+s-1,k}=\sum\limits_{(i_1,\ldots,i_{r-1})\in T_{r-1}}p_{i_1\cdots i_{r-1}i_k'}(\hat{c}_{s,t})a_{s,s+1}^{(i_1)}\cdots a_{r+s-2,r+s-1}^{(i_{r-1})}
\]
for $k=1,\ldots,r$ with $i_k'\neq i_r'$, and
\[
\beta_{s,r+s+t}=\sum\limits_{k=r}^{r+t}\left(\sum\limits_{\substack{s=j_1<\cdots <j_{k+1}=r+s+t\\(j_k,j_{k+1})\neq (r+s-1,r+s+t)\\(i_1,\ldots,i_k)\in T_k}}
p_{i_1\cdots i_k}(\bar{c}_{j_1},\ldots,\bar{c}_{j_{k+1}})a_{j_1j_2}^{(i_1)}\cdots a_{j_{k}j_{k+1}}^{(i_k)}\right).
\]
It follows from both (\ref{e7}) and (\ref{e8}) that

\begin{equation}\label{e9}
p_{s,r+s+t}=f_{s,r+s-1}x_{r+s-1,r+s+t}^{(i_r')}+\sum\limits_{\substack{1\leq k\leq r\\i_k'\neq i_r'}}\beta_{s,r+s-1,k}x_{r+s-1,r+s+t}^{(i_k')}+\beta_{s,r+s+t},
\end{equation}
where $f_{1,r}\in K^*$ and $\beta_{1,r,k}\in K$, $k=1,\ldots,r$ with $i_k'\neq i_r'$, $f_{s,r+s-1}$, $\beta_{s,r+s+t,k}$ are commutative polynomials on some variables with indexes in $\overline{U}_{s,r+s+t}$, where $s\geq 2$ and $k=1,\ldots,r$ with $i_k'\neq i_r'$, and $\beta_{s,r+s+t}$ is a commutative polynomial over $K$ on some variables in the front of $x_{r+s-1,r+s+t}^{(i_r')}$. It is clear that
\[
\beta_{s,r+s+t}=0
\]
in the case of $t=0$.

We define an order of the set
\[
\{(s,r+s+t)~|~1\leq s<r+s+t\leq n\}
\]
as follows:
\begin{enumerate}
\item[(i)] $(s,r+s+t)<(s_1,r+s_1+t_1)$ for $t<t_1$;
\item[(ii)] $(s,r+s+t)<(s_1,r+s_1+t_1)$ for all $t=t_1$ and $s<s_1$.
\end{enumerate}
That is
\begin{equation}\label{e10}
(1,r+1)<\cdots <(n-r,n)<(1,r+2)<\cdots <(n-r-1,n)<\cdots <(1,n).
\end{equation}

We set
\[
W_{s,r+s+t}=\bigcup_{(1,r+1)\leq (i,r+i+j)\leq (s,r+s+t)}U_{i,r+i+j},
\]
where $1\leq s<r+s+t\leq n$ and
\[
\overline{W}_{s,r+s+t}=\bigcup_{(1,r+1)\leq (i,r+i+j)\leq (s,r+s+t)}\overline{U}_{i,r+i+j}.
\]

It is easy to check that
\begin{equation}\label{ww1}
\overline{W}_{s,r+s+t}=\overline{W}_{n-r,n}
\end{equation}
for all $t>0$. Indeed, since $t>0$ we see that $(s,r+s+t)>(n-r,n)$. This implies that $\overline{W}_{s,r+s+t}\supseteq\overline{W}_{n-r,n}$.

Take any $(r+u-1,r+u,i_k')\in \overline{W}_{s,r+s+t}$. It is clear that
\[
(r+u-1,r+u,i_k')\in \overline{U}_{u,r+u}.
\]
Note that
\[
\overline{U}_{u,r+u}\subseteq\overline{W}_{n-r,n}.
\]
We obtain that
\[
(r+u-1,r+u,i_k')\in \overline{W}_{n-r,n}.
\]
This implies that $\overline{W}_{s,r+s+t}\subseteq \overline{W}_{n-r,n}$. Hence,
$\overline{W}_{s,r+s+t}=\overline{W}_{n-r,n}$ as desired.

For any $A'=(a_{s,r+s+t}')\in T_n(K)^{(r-1)}$, where $a_{s,r+s}'\neq 0$ for all $1\leq s<r+s\leq n$, we claim that there exist $c_{r+u-1,r+u+w}^{(i_k')}\in K$ with
\[
(r+u-1,r+u+w,k)\in W_{s,r+s+t}
\]
such that
\[
p_{i,r+i+j}(c_{r+u-1,r+u+w}^{(i_k')})=a_{i,r+i+j}'
\]
for all $(1,r+1)\leq (i,r+i+j)\leq (s,r+s+t)$ and
\[
f_{s',r+s'-v}(c_{r+u-1,r+u}^{(i_k')})\neq 0
\]
for all $f_{s',r+s'-v}$ on some variables with indexes in $\overline{W}_{s,r+s+t}$, where $s'\geq 2$ and $1\leq v\leq min\{s'-1,r-1\}$.

We prove the claim by induction on $(s,r+s+t)$. Suppose first that $(s,r+s+t)=(1,r+1)$. We get from (\ref{e9}) that

\begin{equation}\label{e11}
p_{1,r+1}=f_{1,r}x_{r,r+1}^{(i_r')}+\sum\limits_{\substack{1\leq k\leq r\\i_k\neq i_r'}}\beta_{1,r,k}x_{r,r+1}^{(i_k')},
\end{equation}
where $f_{1,r}\in K^*$, $\beta_{1,r,k}\in K$, $k=1,\ldots,r$ with $i_k'\neq i_r'$.

For any $f_{s',r+s'-v}$ on $x_{r,r+1}^{(i_k')}$, $k=1,\ldots,r$, where $s'\geq 2$ and $1\leq v\leq min\{s'-1,r-1\}$, we get from (\ref{e7}) that  $v=s'-1$ and
\begin{equation}\label{e12}
f_{s',r+s'-v}=f_{s',r}x_{r,r+1}^{(i_{r-s'+1}')}+\sum\limits_{\substack{1\leq k\leq r\\i_k'\neq i_{r-s'+1}'}}\alpha_{s',r,k}x_{r,r+1}^{(i_{k}')},
\end{equation}
where $2\leq s'\leq r$. Note that $f_{s',r}\in K^*$ and $\alpha_{s',r,k}\in K$, $k=1,\ldots,r$ with $i_k'\neq i_{r-s'+1}‘$.

In view of Lemma \ref{L3.1} we get from both (\ref{e11}) and (\ref{e12}) that there exist $c_{r,r+1}^{(i_k')}\in K$, $k=1,\ldots,r$, such that
\[
p_{1,r+1}(c_{r,r+1}^{(i_k')})=a_{1,r+1}'
\]
and
\[
f_{s',r+s'-v}(c_{r,r+1}^{(i_k')})\neq 0,
\]
where $2\leq s'\leq r$ and $v=s'-1$, as desired.

Suppose next that $(s,r+s+t)\neq (1,r+1)$. We rewrite (\ref{e10}) as follows.
\[
(1,r+1)<\cdots<(s_1,r+s_1+t_1)<(s,r+s+t)<\cdots<(1,n).
\]

By induction on $(s_1,r+s_1+t_1)$ we have that there exist $c_{r+u-1,r+u+w}^{(i_k')}\in K$ with
\[
(r+u-1,r+u+w,k)\in W_{s_1,r+s_1+t_1}
\]
such that
\[
p_{i,r+i+j}(c_{r+u-1,r+u+w}^{(i_k')})=a_{i,r+i+j}'
\]
for all $(1,r+1)\leq (i,r+i+j)\leq (s_1,r+s_1+t_1)$ and
\[
f_{s',r+s'-v}(c_{r+u-1,r+u}^{(i_k')})\neq 0
\]
for all $f_{s',r+s'-v}$ on some variables with indexes in $\overline{W}_{s_1,r+s_1+t_1}$, where $s'\geq 2$ and $1\leq v\leq min\{s'-1,r-1\}$.

Suppose first that $t=0$. Note that $s\geq 2$, $s_1=s-1$, and $t_1=0$. It follows from (\ref{e9}) that

\begin{equation}\label{e13}
p_{s,r+s}=f_{s,r+s-1}x_{r+s-1,r+s}^{(i_r')}+\sum\limits_{\substack{1\leq k\leq r\\i_k'\neq i_r'}}\beta_{s,r+s-1,k}x_{r+s-1,r+s}^{(i_k')},
\end{equation}
where $f_{s,r+s-1}$ and $\beta_{s,r+s-1,k}$, $k=1,\ldots,r$ with $i_k'\neq i_r'$, are commutative polynomials over $K$ on some variables with indexes in $\overline{W}_{s-1,r+s-1}$. By induction hypothesis we have that
\[
f_{s,r+s-1}\in K^*\quad\mbox{and}\quad\beta_{s,r+s-1,k}\in K
\]
for all $k=1,\ldots,r$ with $i_k'\neq i_r'$.

Take any $f_{s',r+s'-v}$ on some variables with indexes in $\overline{W}_{s,r+s}$, where $s'\geq 2$ and $1\leq v\leq min\{s'-1,r-1\}$. If $f_{s',r+s'-v}$ is a commutative polynomial over $K$ on some variables with indexes in $\overline{W}_{s-1,r+s-1}$, by induction hypothesis we get that
\[
f_{s',r+s'-v}\in K^*.
\]

Otherwise, $f_{s',r+s'-v}$ is not a commutative polynomial over $K$ on some variables with indexes in $\overline{W}_{s-1,r+s-1}$. Note that
\[
\overline{W}_{s,r+s}\setminus \overline{W}_{s-1,r+s-1}=\left\{(r+s-1,r+s,i_k')~|~k=1,\ldots,r\right\}.
\]

It follows from (\ref{e7}) that $v=s'-s$ and
\begin{equation}\label{e14}
f_{s',r+s'-v}=f_{s',r+s-1}x_{r+s-1,r+s}^{(i_{r-s'+s}')}+\sum\limits_{\substack{1\leq k\leq r\\i_k'\neq i_{r-s'+s}'}}\alpha_{s',r+s-1,k}x_{r+s-1,r+s}^{(i_{k}')}
\end{equation}
for all $1\leq s'-s\leq r-1$, where $f_{s',r+s-1}$ and $\alpha_{s',r+s-1,k}$, $k=1,\ldots,r$ with $i_k'\neq i_{r-s'+s}'$, are commutative polynomial over $K$ on some variables
with indexes in $\overline{W}_{s-1,r+s-1}$. By induction hypothesis we note that
\[
f_{s',r+s-1}\in K^*\quad\mbox{and}\quad \alpha_{s',r+s-1,k}\in K
\]
for all $k=1,\ldots,r$ with $i_k'\neq i_r'$. In view of Lemma \ref{L3.1} we get from both (\ref{e13}) and (\ref{e14}) that there exist $c_{r+s-1,r+s}^{(i_k')}\in K$, $k=1,\ldots,r$, such that
\begin{eqnarray*}
\begin{split}
p_{s,r+s}(c_{r+s-1,r+s}^{(i_k')})=a_{s,r+s}'
\end{split}
\end{eqnarray*}
and
\[
f_{s',r+s'-v}(c_{r+s-1,r+s}^{(i_k')})\neq 0
\]
for all $1\leq s'-s\leq r-1$, as desired.

Suppose next that $t>0$. It follows from (\ref{e9}) that

\begin{equation}\label{e15}
p_{s,r+s+t}=f_{s,r+s-1}x_{r+s-1,r+s+t}^{(i_r')}+\sum\limits_{\substack{1\leq k\leq r\\i_k'\neq i_r'}}\beta_{s,r+s-1,k}x_{r+s-1,r+s+t}^{(i_k')}+\beta_{s,r+s+t},
\end{equation}
where $f_{s,r+s-1}$ and  $\beta_{s,r+s-1,k}$, $k=1,\ldots,r$, are commutative polynomials over $K$ on some variables with indexes on $\overline{W}_{s,r+s+t}$, $\beta_{s,r+s+t}$ is commutative polynomials over $K$ on some variables in the front of $x_{r+s-1,r+s+t}^{(i_r')}$.

We note that
\[
\overline{W}_{s,r+s+t}=\overline{W}_{s_1,r+s_1+t_1}.
\]
Indeed, it follows from (\ref{ww1}) that
\[
\overline{W}_{s,r+s+t}=\overline{W}_{n-r,n}.
\]
Suppose first that $s=1$ and $t=1$ we get that
\[
\overline{W}_{s_1,r+s_1+t_1}=\overline{W}_{n-r,n}.
\]
It implies that
\[
\overline{W}_{s,r+s+t}=\overline{W}_{s_1,r+s_1+t_1},
\]
as desired. Suppose next that either $s>1$ or $t>1$. We note that $t_1\geq 1$. It follows from (\ref{ww1}) that
\[
\overline{W}_{s_1,r+s_1+t_1}=\overline{W}_{n-r,n}.
\]
This implies that
\[
\overline{W}_{s,r+s+t}=\overline{W}_{s_1,r+s_1+t_1},
\]
as desired. It follows that $f_{s,r+s-1}$ and  $\beta_{s,r+s-1,k}$, $k=1,\ldots,r$ with $i_k'\neq i_r'$, are commutative polynomials over $K$ on some variables with indexes in $\overline{W}_{s_1,r+s_1+t_1}$. By induction hypothesis we note that
\[
f_{s,r+s-1}\in K^*\quad\mbox{and}\quad \beta_{s,r+s-1,k}\in K
\]
for all $k=1,\ldots,r$ with $i_k'\neq i_r'$. It is easy to check that
\[
W_{s,r+s+t}\setminus W_{s_1,r+s_1+t_1}=\left\{(r+s-1,r+s+t,i_k')~|~k=1,\ldots, r\right\}.
\]

This implies that $\beta_{s,r+s+t}$ is a commutative polynomial over $K$ on some variables with indexes in $W_{s_1,r+s_1+t_1}$. By induction hypothesis we get that $\beta_{s,r+s+t}\in K$.

Since $f_{s,r+s-1}\in K^*$ we get from (\ref{e15}) that there exist $c_{r+s-1,r+s+t}^{(i_k')}\in K$, $k=1,\ldots,r$, such that
\[
p_{s,r+s+t}(c_{r+s-1,r+s+t}^{(i_k')})=a_{s,r+s+t}'
\]

Take any $f_{s',r+s'-v}$ on some variables with indexes in $\overline{W}_{s,r+s+t}$, where $s'\geq 2$ and $1\leq v\leq min\{s'-1,r-1\}$. Since $\overline{W}_{s,r+s+t}=\overline{W}_{s_1,r+s_1+t_1}$ we get that $f_{s',r+s'-v}$ is a commutative polynomial over $K$ on some variables with indexes in $\overline{W}_{s_1,r+s_1+t_1}$. By induction hypothesis we get that
\[
f_{s',r+s'-v}\in K^*
\]
where $s'\geq 2$ and $1\leq v\leq min\{s'-1,r-1\}$, as desired. This proves the claim.

Let $(s,r+s+t)=(1,n)$. We obtain that there exist $c_{r+u-1,r+u+w}^{(i_k')}\in K$, $k=1,\ldots,r$, with
\[
(r+u-1,r+u+w,k)\in W_{1,n},
\]
such that
\begin{equation}\label{Chen10}
p_{i,r+i+j}(c_{r+u-1,r+u+w}^{(i_k')})=a_{i,r+i+j}'
\end{equation}
for all $(1,r+1)\leq (i,r+i+j)\leq (1,n)$ and
\[
f_{s',r+s'-v}(c_{r+u-1,r+u}^{(i_k')})\neq 0
\]
for all $f_{s',r+s'-v}$ on some variables with indexes in $\overline{W}_{1,n}$, where $s'\geq 2$ and $1\leq v\leq min\{s'-1,r-1\}$.

It follows from both (\ref{e3}) and (\ref{Chen10}) that
\[
p(u_1,\ldots,u_m)=(p_{s,r+s+t})=(a_{s,r+s+t}')=A'.
\]
This implies that $A'\in p(T_n(K))$. The proof of the result is complete.
\end{proof}

The following result is an interesting result, which will be used in the proof of the main result of the paper.

\begin{lemma}\label{L10}
Let $n\geq 4$ and $m\geq 1$ be integers. Let $p(x_1,\ldots,x_m)$ be a
polynomial with zero constant term in non-commutative variables over an infinite field $K$. Suppose that ord$(p)=n-2$. We have that $p(T_n(K))=T_n(K)^{(n-3)}$.
\end{lemma}

\begin{proof}
In view of Lemma \ref{L2.1}(ii) we note that $p(T_n(K))\subseteq T_n(K)^{(n-3)}$. It suffices to prove that $T_n(K)^{(n-3)}\subseteq p(T_n(K))$.

For any $u_i=(a_{jk}^{(i)})\in T_n(K)$, $i=1,\ldots,m$, we get from (\ref{e2}) that
\begin{equation}\label{Jia1}
p(u_1,\ldots,u_m)=\left(
\begin{array}{ccccc}
0 & 0 & \ldots & p_{1,n-1} & p_{1n}\\
0 & 0 & \ldots & 0 & p_{2n}\\
 \vdots & \vdots & \ddots & \vdots & \vdots\\
 0 & 0 & \ldots & 0 & 0
\end{array} \right),
\end{equation}
where
\[
\left\{
\begin{aligned}
p_{1,n-1}&=\sum\limits_{(i_1,\ldots,i_{n-2})\in T_{n-2}}p_{i_1\cdots i_{n-2}}(\bar{a}_{11},\ldots,\bar{a}_{n-1,n-1})a_{12}^{(i_1)}\cdots a_{n-2,n-1}^{(i_{n-2})};\\
p_{2n}&=\sum\limits_{(i_1,\ldots,i_{n-2})\in T_{n-2}}
p_{i_1\cdots i_{n-2}}(\bar{a}_{22},\ldots,\bar{a}_{n,n})a_{23}^{(i_1)}\cdots a_{n-1,n}^{(i_{n-2})};\\
p_{1n}&=\sum\limits_{\substack{1=j_1<\cdots <j_{n-1}=n\\(i_1,\ldots,i_{n-2})\in T_{n-2}}}
p_{i_1\cdots i_{n-2}}(\bar{a}_{j_1j_1},\ldots,\bar{a}_{j_{n-1}j_{n-1}})a_{j_1j_2}^{(i_1)}\cdots a_{j_{n-2}j_{n-1}}^{(i_{n-2})}\\
&\ \ \ +\sum\limits_{(i_1,\ldots,i_{n-1})\in T_{n-1}}p_{i_1\cdots i_{n-1}}(\bar{a}_{11},\ldots,\bar{a}_{nn})a_{12}^{(i_1)}\cdots a_{n-1,n}^{(i_{n-1})}.
\end{aligned}
\right.
\]

In view of Lemma \ref{L2.1}(iii) we have that
\[
p_{i_1',\ldots,i_{n-2}'}(K)\neq\{0\},
\]
for some $i_1',\ldots,i_{n-2}'\in\{1,\ldots,m\}$. It follows from Lemma \ref{L2.3} that
there exist $\bar{b}_1,\ldots,\bar{b}_n\in K^m$ such that
\[
p_{i_1',\ldots,i_{n-2}'}(\bar{b}_{j_1},\ldots,\bar{b}_{j_{n-1}})\neq 0
\]
for all $1\leq j_1<\cdots <j_{n-1}\leq n$.

For any $A'=(a_{s,n-2+s+t}')\in T_n(K)^{(n-3)}$, where $1\leq s<n-2+s+t\leq n$, we claim that there exist $u_i=(a_{jk}^{(i)})\in T_n(K)$, $i=1,\ldots,m$, such that
\[
p(u_1,\ldots,u_m)=(p_{s,n-2+s+t})=A'.
\]
That is
\[
\left\{
\begin{aligned}
p_{1,n-1}&=a_{1,n-1}';\\
p_{2n}&=a_{2n}';\\
p_{1n}&=a_{1n}'.
\end{aligned}
\right.
\]

We prove the claim by the following two cases:\\

\emph{Case 1.} Suppose that $a_{1,n-1}'\neq 0$. We take
\[
\left\{
\begin{aligned}
\bar{a}_{ii}&=\bar{b}_i, \quad\mbox{for all $i=1,\ldots,n$};\\
a_{12}^{(i_1')}&=x_{12}^{(i_1')};\\
a_{12}^{(k)}&=0\quad\mbox{for all $k=1,\ldots,m$ with $k\neq i_1'$};\\
a_{n-1,n}^{(i_{n-2}')}&=x_{n-1,n}^{(i_{n-2}')};\\
a_{n-1,n}^{(k)}&=0\quad\mbox{for all $k=1,\ldots,m$ with $k\neq i_{n-2}'$};\\
a_{n-2,n}^{(i_{n-2}')}&=x_{n-2,n}^{(i_{n-2}')};\\
a_{j,j+2}^{(i)}&=0\quad\mbox{for all $1\leq j<j+2\leq n$ with $(j,j+2,i)\neq (n-2,n,i_{n-2}')$}.
\end{aligned}
\right.
\]
It follows from (\ref{Jia1}) that

\begin{equation}\label{Jia2}
\left\{
\begin{aligned}
p_{1,n-1}&=\left(\sum\limits_{\substack{(i_2,\ldots,i_{n-2})\in T_{n-3}}}p_{i_1'i_2\cdots i_{n-2}}(\bar{b}_{1},\ldots,\bar{b}_{n-1})a_{23}^{(i_2)}\cdots a_{n-2,n-1}^{(i_{n-2})}\right)x_{12}^{(i_1')};\\
p_{2n}&=\left(\sum\limits_{\substack{(i_1,\ldots,i_{n-3})\in T_{n-3}}}p_{i_1\cdots i_{n-3}i_{n-2}'}(\bar{b}_{2},\ldots,\bar{b}_{n})a_{23}^{(i_1)}\cdots a_{n-2,n-1}^{(i_{n-3})}\right)x_{n-1,n}^{(i_{n-2}')};\\
p_{1n}&=\left(\sum\limits_{\substack{(i_2,\ldots,i_{n-3})\in T_{n-4}}}p_{i_1'i_2\cdots i_{n-3}i_{n-2}'}(\bar{b}_{1},\ldots,\bar{b}_{n-2},\bar{b}_{n})a_{23}^{(i_2)}\cdots a_{n-3,n-2}^{(i_{n-3})}\right)x_{12}^{(i_1')}x_{n-2,n}^{(i_{n-2}')}\\
&\ \ \ +\left(\sum\limits_{(i_2,\ldots,i_{n-2})\in T_{n-3}}p_{i_1'i_2\cdots i_{n-2}i_{n-2}'}(\bar{b}_{1},\ldots,\bar{b}_{n})a_{23}^{(i_2)}\cdots a_{n-2,n-1}^{(i_{n-2})}\right)x_{12}^{(i_1')}x_{n-1,n}^{(i_{n-2}')}.
\end{aligned}
\right.
\end{equation}

We set
\begin{equation}\label{Jia5}
\left\{
\begin{aligned}
f_{1,n-1}&=\sum\limits_{\substack{(i_2,\ldots,i_{n-2})\in T_{n-3}}}p_{i_1'i_2\cdots i_{n-2}}(\bar{b}_{1},\ldots,\bar{b}_{n-1})a_{23}^{(i_2)}\cdots a_{n-2,n-1}^{(i_{n-2})}\\
f_{2n}&=\sum\limits_{\substack{(i_1,\ldots,i_{n-3})\in T_{n-3}}}p_{i_1\cdots i_{n-3}i_{n-2}'}(\bar{b}_{2},\ldots,\bar{b}_{n})a_{23}^{(i_1)}\cdots a_{n-2,n-1}^{(i_{n-3})}\\
f_{1n}&=\sum\limits_{\substack{(i_2,\ldots,i_{n-3})\in T_{n-4}}}p_{i_1'i_2\cdots i_{n-3}i_{n-2}'}(\bar{b}_{1},\ldots,\bar{b}_{n-2},\bar{b}_{n})a_{23}^{(i_2)}\cdots a_{n-3,n-2}^{(i_{n-3})}
\end{aligned}
\right.
\end{equation}
and
\begin{eqnarray*}
\begin{split}
V_{1,n-1}&=\{(i,i+1,k)~|~i=2,\ldots,n-2,k=1,\ldots,m\};\\
V_{2n}&=\{(i,i+1,k)~|~i=2,\ldots,n-2,k=1,\ldots,m\};\\
V_{1n}&=\{(i,i+1,k)~|~i=2,\ldots,n-3,k=1,\ldots,m\}.
\end{split}
\end{eqnarray*}

Note that $V_{1,n-1},V_{2n}, V_{1n}$ are the index sets of $f_{1,n-1},f_{2n},f_{1n}$, respectively.

We claim that $f_{1,n-1},f_{2n},f_{1n}\neq 0$. Indeed, we take $a_{jk}^{(i)}\in K$, $(j,k,i)\in V_{1,n-1}$ such that
\[
\left\{
\begin{aligned}
a_{s,s+1}^{(i_s')}&=1\quad\mbox{for all $s=2,\ldots,n-2$};\\
a_{jk}^{(i)}&=0\quad\mbox{otherwise}.
\end{aligned}
\right.
\]
It follows from (\ref{Jia5}) that
\[
f_{1,n-1}(a_{jk}^{(i)})=p_{i_1'\cdots i_{n-2}'}(\bar{b}_{1},\ldots,\bar{b}_{n-1})\neq 0
\]
as desired. Next, we take $a_{jk}^{(i)}\in K$, $(j,k,i)\in V_{2n}$ such that
\[
\left\{
\begin{aligned}
a_{s,s+1}^{(i_{s-1}')}&=1\quad\mbox{for all $s=2,\ldots,n-2$};\\
a_{jk}^{(i)}&=0\quad\mbox{otherwise}.
\end{aligned}
\right.
\]
It follows from (\ref{Jia5}) that
\[
f_{2n}(a_{jk}^{(i)})=p_{i_1'\cdots i_{n-2}'}(\bar{b}_{2},\ldots,\bar{b}_{n})\neq 0
\]
as desired. Finally, we take $a_{jk}^{(i)}\in K$, $(j,k,i)\in V_{1n}$ such that
\[
\left\{
\begin{aligned}
a_{s,s+1}^{(i_s')}&=1\quad\mbox{for all $s=2,\ldots,n-3$};\\
a_{jk}^{(i)}&=0\quad\mbox{otherwise}.
\end{aligned}
\right.
\]
It follows from (\ref{Jia5}) that
\[
f_{1n}(a_{jk}^{(i)})=p_{i_1'\cdots i_{n-2}'}(\bar{b}_{1},\ldots,\bar{b}_{n-2},\bar{b}_n)\neq 0
\]
as desired. In view of Lemma \ref{L2.4} we get that there exist $a_{jk}^{(i)}\in K$, where $(j,k,i)\in V_{1,n-1}\cup V_{2n}\cup V_{1n}$ such that
\[
\left\{
\begin{aligned}
f_{1,n-1}(a_{jk}^{(i)})&\neq 0;\\
f_{2n}(a_{jk}^{(i)})&\neq 0;\\
f_{1n}(a_{jk}^{(i)})&\neq 0.
\end{aligned}
\right.
\]

We set
\[
\alpha=\sum\limits_{(i_2,\ldots,i_{n-2})\in T_{n-3}}p_{i_1'i_2\cdots i_{n-2}i_{n-2}'}(\bar{b}_{1},\ldots,\bar{b}_{n})a_{23}^{(i_2)}\cdots a_{n-2,n-1}^{(i_{n-2})}.
\]
It follows from (\ref{Jia2}) that
\begin{equation}\label{Jia6}
\left\{
\begin{aligned}
p_{1,n-1}&=f_{1,n-1}x_{12}^{(i_1')};\\
p_{2n}&=f_{2n}x_{n-1,n}^{(i_{n-2}')};\\
p_{1n}&=f_{1n}x_{12}^{(i_1')}x_{n-2,n}^{(i_{n-2}')}+\alpha x_{12}^{(i_1')}x_{n-1,n}^{(i_{n-2}')}.
\end{aligned}
\right.
\end{equation}

We take
\[
\left\{
\begin{aligned}
x_{12}^{(i_1')}&=f_{1,n-1}^{-1}a_{1,n-1}';\\
x_{n-1,n}^{(i_{n-2}')}&=f_{2n}^{-1}a_{2n}';\\
x_{n-2,n}^{(i_{n-2}')}&=f_{1n}^{-1}f_{1,n-1}(a_{1,n-1}')^{-1}\left(a_{1n}'-\alpha f_{1,n-1}^{-1}a_{1,n-1}'f_{2n}^{-1}a_{2n}'\right).
\end{aligned}
\right.
\]
It follows from (\ref{Jia6}) that
\[
\left\{
\begin{aligned}
p_{1,n-1}&=a_{1,n-1}';\\
p_{2n}&=a_{2n}';\\
p_{1n}&=a_{1n}',
\end{aligned}
\right.
\]
as desired.\\

\emph{Case 2.} Suppose that $a_{1,n-1}'=0$.  We take
\[
\left\{
\begin{aligned}
\bar{a}_{ii}&=\bar{b}_i, \quad\mbox{for all $i=1,\ldots,n$};\\
a_{12}^{(k)}&=0\quad\mbox{for all $k=1,\ldots,m$};\\
a_{23}^{(i_{1}')}&=x_{23}^{(i_{1}')};\\
a_{23}^{(k)}&=0\quad\mbox{for all $k=1,\ldots,m$ with $k\neq i_{1}'$};\\
a_{13}^{(i_{1}')}&=x_{13}^{(i_{1}')};\\
a_{j,j+2}^{(k)}&=0\quad\mbox{for all $1\leq j<j+2\leq n$ with $(j,j+2,k)\neq (1,3,i_1')$}.
\end{aligned}
\right.
\]

It follows from (\ref{Jia1}) that
\begin{equation}\label{Jia7}
\left\{
\begin{aligned}
p_{1,n-1}&=0;\\
p_{2n}&=\left(\sum\limits_{\substack{(i_2,\ldots,i_{n-2})\in T_{n-3}}}p_{i_1'i_2\cdots i_{n-2}}(\bar{b}_{2},\ldots,\bar{b}_{n})a_{34}^{(i_2)}\cdots a_{n-1,n}^{(i_{n-2})}\right)x_{23}^{(i_{1}')};\\
p_{1n}&=\left(\sum\limits_{\substack{(i_2,\ldots,i_{n-2})\in T_{n-3}}}p_{i_1'i_2\cdots i_{n-2}}(\bar{b}_{1},\bar{b}_3,\ldots,\bar{b}_{n})a_{34}^{(i_2)}\cdots a_{n-1,n}^{(i_{n-2})}\right)x_{13}^{(i_1')}.
\end{aligned}
\right.
\end{equation}

We set
\begin{equation}\label{Jia8}
\left\{
\begin{aligned}
g_{2n}&=\sum\limits_{\substack{(i_2,\ldots,i_{n-2})\in T_{n-3}}}p_{i_1'i_2\cdots i_{n-2}}(\bar{b}_{2},\ldots,\bar{b}_{n})a_{34}^{(i_2)}\cdots a_{n-1,n}^{(i_{n-2})};\\
g_{1n}&=\sum\limits_{\substack{(i_2,\ldots,i_{n-2})\in T_{n-3}}}p_{i_1'i_2\cdots i_{n-2}}(\bar{b}_{1},\bar{b}_3,\ldots,\bar{b}_{n})a_{34}^{(i_2)}\cdots a_{n-1,n}^{(i_{n-2})}
\end{aligned}
\right.
\end{equation}
and
\[
V=\{(i,i+1,k)~|~i=3,\ldots,n-1,k=1,\ldots,m\}.
\]

Note that $V$ is the index set of $g_{2n}$ and $g_{1n}$. We claim that $g_{2n},g_{1n}\neq 0$. Indeed, we take $a_{jk}^{(i)}\in K$, $(j,k,i)\in V$ such that
\[
\left\{
\begin{aligned}
a_{s,s+1}^{(i_{s-1}')}&=1\quad\mbox{for all $s=3,\ldots,n-1$};\\
a_{jk}^{(i)}&=0\quad\mbox{otherwise}.
\end{aligned}
\right.
\]

It follows from (\ref{Jia8}) that

\begin{eqnarray*}
\begin{split}
g_{2n}&=p_{i_1'\cdots i_{n-2}'}(\bar{b}_{2},\ldots,\bar{b}_{n})\neq 0;\\
g_{1n}&=p_{i_1'\cdots i_{n-2}'}(\bar{b}_{1},\bar{b}_3,\ldots,\bar{b}_{n})\neq 0.
\end{split}
\end{eqnarray*}
as desired. It follows from (\ref{Jia7}) that

\begin{equation}\label{Jia9}
\left\{
\begin{aligned}
p_{1,n-1}&=0;\\
p_{2n}&=g_{2n}x_{23}^{(i_{1}')};\\
p_{1n}&=g_{1n}x_{13}^{(i_1')}.
\end{aligned}
\right.
\end{equation}

We take
\[
\left\{
\begin{aligned}
x_{23}^{(i_1')}&=g_{2n}^{-1}a_{2n}';\\
x_{13}^{(i_1')}&=g_{1n}^{-1}a_{1n}'.
\end{aligned}
\right.
\]
It follows from (\ref{Jia9}) that
\[
\left\{
\begin{aligned}
p_{1,n-1}&=0;\\
p_{2n}&=a_{2,n}';\\
p_{1n}&=a_{1n}',
\end{aligned}
\right.
\]
as desired. We obtain that
\[
p(u_1,\ldots,u_m)=(p_{s,n-2+s+t})=(a_{s,n-2+s+t}')=A'.
\]
This implies that $T_n(K)^{(n-3)}\subseteq p(T_n(K))$. Hence $p(T_n(K))=T_n(K)^{(n-3)}$. The proof of the result is complete.
\end{proof}

We are ready to give the proof of the main result of the paper.

\begin{proof}[The proof of Theorem \ref{T1}]
For any $A=(a_{s,r+s+t})\in T_n(K)^{(r-1)}$, we consider the following polynomials:
\[
\left\{
\begin{aligned}
f_{s,r+s}(x_{s,r+s})&=a_{s,r+s}-x_{s,r+s};\\
g_{s,r+s}(x_{s,r+s})&=x_{s,r+s}
\end{aligned}
\right.
\]
for all $1\leq s<r+s\leq n$. It is clear that both $f_{s,r+s}$ and $g_{s,r+s}$ are nonzero commutative polynomials over $K$, where $1\leq s<r+s\leq n$. It follows from Lemma \ref{L2.4} that there exist $b_{s,r+s}\in K$, $1\leq s<r+s\leq n$, such that
\[
\left\{
\begin{aligned}
f_{s,r+s}(b_{s,r+s})&\neq 0;\\
g_{s,r+s}(b_{s,r+s})&\neq 0
\end{aligned}
\right.
\]
for all $1\leq s<r+s\leq n$. That is
\[
\left\{
\begin{aligned}
a_{s,r+s}-b_{s,r+s}&\neq 0;\\
b_{s,r+s}&\neq 0
\end{aligned}
\right.
\]
for all $1\leq s<r+s\leq n$. We set
\[
c_{s,r+s}=a_{s,r+s}-b_{s,r+s}
\]
for all $1\leq s<r+s\leq n$ and
\[
\left\{
\begin{aligned}
b_{s,r+s+t}&=a_{s,r+s+t};\\
c_{s,r+s+t}&=0
\end{aligned}
\right.
\]
for all $1\leq s<r+s+t\leq n$ and $t>0$. We set
\[
B=(b_{s,r+s+t})\quad\mbox{and}\quad C=(c_{s,r+s+t}).
\]
It is clear that $B,C\in T_n(K)^{(r-1)}$ and
\[
A=B+C.
\]

In view of Lemma \ref{L3.2}, we get that there exist $u_i,v_i\in T_n(K)$, $i=1,\ldots,m$, such that
\[
p(u_1,\ldots,u_m)=B
\]
and
\[
p(v_1,\ldots,v_m)=C.
\]
It follows that
\[
p(u_1,\ldots,u_m)+p(v_1,\ldots,v_m)=A.
\]
This implies that
\[
T_n(K)^{(r-1)}\subseteq p(T_n(K))+p(T_n(K)).
\]
In view of Lemma \ref{L2.1}(ii) we note that $p(T_n(K))\subseteq T_n(K)^{(r-1)}$. Since $T_n(K)^{(r-1)}$ is a subspace of $T_n(K)$ we get that
\[
 p(T_n(K))+p(T_n(K))\subseteq T_n(K)^{(r-1)}.
 \]
We obtain that
\[
 p(T_n(K))+p(T_n(K))=T_n(K)^{(r-1)}.
 \]

In particular, if $r=n-2$ we get from Lemma \ref{L10} that
\[
p(T_n(K))=T_n(K)^{(n-3)}.
\]
The proof of the result is complete.
\end{proof}

We conclude the paper with following example.

\begin{example}\label{EX}
Let $n\geq 5$ and $1<r<n-2$ be integers. Let $K$ be an infinite field. Let
\[
p(x,y)=[x,y]^r.
\]
We have that ord$(p)=r$ and $p(T_n(K))\neq T_n(K)^{(r-1)}$.
\end{example}

\begin{proof}
It is easy to check that $p(T_r(K))=\{0\}$. Set
\[
f(x,y)=[x,y].
\]
Note that $f$ is a multilinear polynomial over $K$. It is clear that ord$(f)=1$. In  view of \cite[Theorem 4.3]{Mello} or \cite[Theorem 1.1]{Wang2022} we have that
\[
f(T_{r+1}(K))=T_{r+1}(K)^{(0)}.
\]
It implies that there exist $A,B\in T_{r+1}(K)$ such that
\[
[A,B]=e_{12}+e_{23}+\cdots +e_{r,r+1}.
\]
We get that
\[
p(A,B)=[A,B]^r=e_{1,r+1}\neq 0.
\]
This implies that $p(T_{r+1}(K))\neq\{0\}$. We obtain that ord$(p)=r$.

Suppose on contrary that $p(T_n(K))=T_n(K)^{(r-1)}$ for some $n\geq 5$ and $1<r<n-2$. For $A'=e_{1,r+1}+e_{3,r+3}\in T_n(K)^{(r-1)}$, we get that there exists $B,C\in T_n(K)$ such that
\[
p(B,C)=[B,C]^r=A'.
\]
It is clear that $[B,C]\in T_n(K)^{(0)}$. We set
\[
A=[B,C]=(a_{s,1+s+t}).
\]
It follows that
\[
A^r=A'.
\]
We get from the last relation that
\[
\left\{
\begin{aligned}
(a_{12}a_{23}\cdots a_{r,r+1})e_{1,r+1}&=e_{1,r+1};\\
(a_{23}a_{34}\cdots a_{r+1,r+2})e_{2,r+2}&=0;\\
(a_{34}a_{45}\cdots a_{r+2,r+3})e_{3,r+3}&=e_{3,r+3}.
\end{aligned}
\right.
\]
This is a contradiction. We obtain that $p(T_n(K))\neq T_n(K)^{(r-1)}$ for all $n\geq 5$ and $1<r<n-2$. This proves the result.
\end{proof}

We remark that \cite[Example 5.7]{PP} is a special case of Example \ref{EX} ($r=2$ and $n=5$).\\

\end{document}